\documentclass[reqno]{amsart}
\usepackage{amssymb}
\usepackage{amsfonts}

\newtheorem{theorem}{Theorem}
\theoremstyle{plain}

\newtheorem{claim}{Claim}

\newtheorem{example}{Example}

\newtheorem{lemma}{Lemma}

\newtheorem{proposition}{Proposition}

\numberwithin{equation}{section}

\begin{document}
\title[self-similar sets with exact overlaps]{Lipschitz equivalence of
self-similar sets with exact overlaps}
\author{kan Jiang}
\address{Department of Mathematics, Ningbo University, Ningbo 315211, P. R.
China}
\email{jiangkan@nbu.edu.cn}
\author{songjing Wang}
\address{Department of Mathematics, Ningbo University, Ningbo 315211, P. R.
China}
\email{wangsongjing@nbu.edu.cn}
\author{Lifeng Xi}
\address{Department of Mathematics, Ningbo University, Ningbo 315211, P. R.
China}
\email{xilifeng@nbu.edu.cn}
\thanks{Lifeng Xi is the corresponding author. The work is supported by
National Natural Science Foundation of China (Nos. 11771226, 11701302,
11371329, 11471124) and Philosophical and Social Science Planning of
Zhejiang Province (No. 17NDJC108YB). The work is also supported by K.C. Wong
Magna Fund in Ningbo University.}

\begin{abstract}
In this paper, we study a class $\mathcal{A}(\lambda ,n,m)$ of\ self-similar
sets with $m$ exact overlaps generated by $n$ similitudes of the same ratio $%
\lambda .$ We obtain a necessary condition for a self-similar set in $%
\mathcal{A}(\lambda ,n,m)$ to be Lipschitz equivalent to a self-similar set
satisfying the strong separation condition, i.e., there exists an integer $%
k\geq 2$ such that $x^{2k}-mx^{k}+n$ is reducible, in particular, $m$
belongs to $\{a^{i}:a\in \mathbb{N}$ with $i\geq 2\}.$
\end{abstract}

\maketitle





\section{Introduction}

\bigskip

Recall that a compact subset $K$ of Euclidean space is said to be a
self-similar set \cite{H}, if $K=\cup _{i=1}^{n}S_{i}(K)$ is generated by
contractive similitudes $\{S_{i}\}_{i}$ with ratio set $\{r_{i}\}_i\subset
(0,1)$ satisfying $|S_{i}(x)-S_{i}(y)|=r_{i}|x-y|$ for all $x,y.$ The
classical dimension result under the open set condition (\textbf{OSC}) is
\begin{equation}
\dim _{H}K=s\text{ with }\sum\nolimits_{i=1}^{n}(r_{i})^{s}=1.  \label{dim}
\end{equation}
In particular, $K$ is said to be \textbf{dust-like} when the strong
separation condition (\textbf{SSC}) holds, i.e., $S_{i}(K)\cap
S_{j}(K)=\emptyset $ for all $i\neq j,$ then the open set condition holds
and thus (\ref{dim}) is valid.

The self-similar sets with overlaps have complicated structures, for
example, Hochman \cite{Hoch} studied the self-similar sets%
\begin{equation*}
E_{\theta}=E_{\theta}/3\cup (E_{\theta}/3+\theta/3)\cup (E_{\theta}/3+2/3)
\end{equation*}
and obtained $\dim _{H}E_{\theta}=1$ for any $\theta$ irrational. If $\theta
$ is rational, Kenyon \cite{K} obtained that the \textbf{OSC} is fulfilled
for $E_{\theta}$ if and only if $\theta=p/q\in \mathbb{Q}$ with $p\equiv
q\not\equiv 0$ (mod3). Rao and Wen \cite{R} also discussed the structure of $%
E_{\theta}$ with $\theta\in \mathbb{Q}$ using the key idea \textquotedblleft
graph-directed structure\textquotedblright\ introduced by Mauldin and
Williams \cite{M}.

\medskip Recently, Jiang, Wang and Xi \cite{jiang} investigated a class $%
\mathcal{A}(\lambda ,n,m)$ of self-similar sets with exact overlaps where $%
\lambda \in (0,1)$ and $m,n\in \mathbb{N}$ with $1\leq m\leq n-2$. Let%
\textbf{\ }$f_{i}(x)=\lambda x+b_{i}$ with $0=b_{1}<b_{2}<\cdots
<b_{n}=1-\lambda .$ Write $I=[0,1]$ and $I_{i}=f_{i}(I).$ Assume that
\begin{equation*}
\frac{\left\vert I_{i}\cap I_{i+1}\right\vert }{|I_{i}|}\in \{0,\lambda \}%
\text{ if }I_{i}\cap I_{i+1}\neq \emptyset ,\text{ and }\sharp \{i:\frac{%
\left\vert I_{i}\cap I_{i+1}\right\vert }{|I_{i}|}=\lambda \}=m.
\end{equation*}%
We call $E=\cup _{i=1}^{n}f_{i}(E)$ a self-similar set with exact overlap,
denoted by $E\in \mathcal{A}(\lambda ,n,m).$ It is proved in \cite{jiang}
that $\dim _{H}E=\frac{\log \beta }{-\log \lambda }$ where the P.V. number $%
\beta >1$ is a root of the irreducible polynomial $x^{2}-nx+m=(x-\beta
)(x-\beta ^{\prime })$ with $|\beta ^{\prime }|<1<\beta .$

\medskip

In this paper, we will compare self-similar sets in $\mathcal{A}(\lambda
,n,m)$ with dust-like self-similar sets in terms of Lipschitz equivalence.

Two compact subsets $X_{1}$, $X_{2}$ of Euclidean spaces are said to be
Lipschitz equivalent, denoted by $X_1\simeq X_2$, if there is a bijection $f:
$ $X_{1}\rightarrow X_{2}$ and a constant $C>0$ such that for all $x$, $y\in
X_{1},$%
\begin{equation*}
C^{-1}|x-y|\leq |f(x)-f(y)|\leq C|x-y|.
\end{equation*}
Cooper and Pignataro \cite{CP}, Falconer and Marsh \cite{F}, David and
Semmes \cite{GS} and Wen and Xi \cite{WX} showed that two self-similar sets
need not be Lipschitz equivalent although they have the same Hausdorff
dimension.

We concern the Lipschitz equivalence between two self-similar sets with the
\textbf{SSC} and with overlaps respectively.

(1) David and Semmes \cite{GS} posed the $\{1,3,5\}$-$\{1,4,5\}$ problem.
Let $H_{1}=(H_{1}/5)\cup (H_{1}+2/5)\cup (H_{1}+4/5)$ and $%
H_{2}=(H_{2}/5)\cup (H_{2}+3/5)\cup (H_{2}+4/5)\ $be $\{1,3,5\}$, $\{1,4,5\}$
self-similar sets respectively. The problem asks about the Lipschitz
equivalence between $H_{1}$ (with the \textbf{SSC}) and $H_{2}$ (with the
touched structure). Rao, Ruan and Xi \cite{RRX} proved that $H_{1}$ and $%
H_{2}$ are Lipschitz equivalent.

(2) Guo et al. \cite{X} studied the Lipschitz equivalence for $%
K_{n}=(\lambda K_{n})\cup (\lambda K_{n}+\lambda ^{n}(1-\lambda ))\cup
(\lambda K_{n}+1-\lambda )$ with overlaps and proved that $K_{n}\simeq K_{m}$
for all $n,m\geq 1.$ In particular, for $n=1,$ $K_{1}\in \mathcal{A}(\lambda
,3,1)$ is Lipschitz equivalent to a dust-like set $F=(\lambda F)\cup
(\lambda ^{1/2}F+1-\lambda ^{1/2})$.

We will state our main result.

\begin{theorem}
Suppose $E\in \mathcal{A}(\lambda ,n,m)$ and $P(x)=x^{2}-nx+m.$ If there is
a dust-like self-similar set $F$ such that $E\simeq F,$\textbf{\ }then there
exists an integer $k\geq 2$ such that
\begin{equation*}
P(x^{k})=x^{2k}-nx^{k}+m\text{ is reducible in }\mathbb{Z}[x].
\end{equation*}%
In particular, we have
\begin{equation*}
m\in \{a^{i}\ |\ a\in \mathbb{N}\text{ and }i\in \mathbb{N}\text{ with }%
i\geq 2\}.
\end{equation*}
\end{theorem}

By this theorem, if $m\in \{2,3,5,6,7,10,11,12,13,14,15,17,\cdots \},$ then
we cannot find a dust-like self-similar set to be Lipschitz equivalent to $%
E\in \mathcal{A}(\lambda ,n,m).$

\begin{example}
For $n=3$ and $m=1,$ we have $P(x)=x^{2}-3x+1$ and an example $K_{1}\simeq
F=(\lambda F)\cup (\lambda ^{1/2}F+1-\lambda ^{1/2})$ in \cite{X} as above.
Now, $P(x^{2})=(x^{2}-x-1)(x^{2}+x-1)$ is reducible and $1\in \{a^{i}\ |\
a\in \mathbb{N}$ and $i\in \mathbb{N}$ with $i\geq 2\}.$
\end{example}

\medskip

The paper is organized as follows. In Section 2 we show any self-similar set
in $\mathcal{A}(\lambda, n, m)$ has graph-directed structure and obtain the
logarithmic commensurability of ratios for the dust-like self-similar set by
the approach of Falconer and Marsh \cite{F}. Using the dimension polynomials
and their irreducibility, we give the proof of Theorem 1 in Section 3.

\bigskip

\section{Logarithmic Commensurability of Ratios}

At first, we show that any self-similar set with exact overlaps will
generate a graph-directed construction.

\begin{lemma}
\label{l:graph}There are graph-directed sets $\{E_{i}\}_{i=1}^{u}$ with
ratio $\lambda $ satisfying the \textbf{SSC} and $E_{1}=E$.
\end{lemma}

\begin{proof}
Consider the set $G$ in the following form
\begin{equation*}
G=\cup _{i=1}^{k}(E+a_{i})\text{ with }0=a_{1}<a_{2}<\cdots <a_{k}\text{ and
}k\leq n-1
\end{equation*}%
such that $(I+a_{i})\cap (I+a_{i+1})\neq \emptyset$ with $I=[0,1]$ for all $%
i\leq k-1$ satisfying
\begin{equation*}
|(I+a_{i})\cap (I+a_{i+1})|=0\text{ or }\lambda.
\end{equation*}

Let $\mathcal{G}$ be the collection of all sets in the form as above.

For every $G\in \mathcal{G},$ considering the natural decomposition at the
touched point ($|(I+a_{i})\cap (I+a_{i+1})|=0$) or on the exact overlapping (%
$|(I+a_{i})\cap (I+a_{i+1})|=\lambda$), we have the decomposition
\begin{equation*}
G=\bigcup _{G^{\prime }\in \mathcal{G}}\bigcup _{i}(\lambda G^{\prime
}+b_{i,G,G^{\prime }})  \label{ddd}
\end{equation*}%
which is a disjoint union. That means we obtain a graph directed
construction satisfying the \textbf{SSC}. In fact, we only need to choose a
subgraph generated by $E$ with $k=1.$
\end{proof}

The main result of this section is the following Proposition 1. We will use
the approach by Falconer and Marsh \cite{F}. In \cite{F}, the authors
discussed the dust-like self-similar sets, now we will deal with the
graph-directed sets.

\begin{proposition}
\label{P:1}Suppose $E\in\mathcal{A}(\lambda ,n,m)$ and $F=\cup
_{j=1}^{t}g_{j}(F)$ is a dust-like self-similar set such that $E\simeq F.$
Assume $r_{j}$ is the contractive ratio of $g_{j}$ for any $j.$ Then there
is a ratio $r\in (0,1)$ and positive integers $k$ and $k_{1}\leq k_{2}\leq
\cdots \leq k_{t}$ such that
\begin{equation*}
\lambda =r^{k},\text{ }r_{1}=r^{k_{1}},r_{2}=r^{k_{2}},\cdots
,r_{t}=r^{k_{t}}.
\end{equation*}
\end{proposition}

Without loss of generality, we only need to show that%
\begin{equation*}
\frac{\log r_{j}}{\log \lambda }\in \mathbb{Q},
\end{equation*}%
or $\frac{\log (r_{j})^{s}}{\log \lambda ^{s}}\in \mathbb{Q}$ with $s=\dim
_{H}E=\dim _{H}F.$ Suppose $f:F\rightarrow E$ is a bi-Lipschitz bijection
and $c\geq 1$ is a constant satisfying
\begin{equation*}
c^{-1}|x-y|\leq |f(x)-f(y)|\leq c|x-y|\text{ for all }x,y\in F.
\end{equation*}%
Denote $\Sigma ^{\ast }=\bigcup\nolimits_{k\geq 0}\{1,\cdots ,t\}^{k}.$ For
any $\mathbf{j}=j_{1}\cdots j_{k}\in \Sigma ^{\ast },$ we write $F_{\mathbf{j%
}}=g_{j_{1}\cdots j_{k}}(F).$

Suppose $\mathbf{e}$ is an admissible path of length $|\mathbf{e}|$ in the
directed graph beginning at vertex $v=b(\mathbf{e}),$ then
\begin{equation}  \label{333}
|E_{\mathbf{e}}|=\lambda ^{|\mathbf{e}|}|E_{v}|\text{ and }\mathcal{H}%
^{s}(E_{\mathbf{e}})=\lambda ^{s|\mathbf{e}|}\mathcal{H}^{s}(E_{v})=\lambda
^{s|\mathbf{e}|}\mathcal{H}^{s}(E_{b(\mathbf{e})}).
\end{equation}

Because of the \textbf{SSC }on $F,$ we assume that there is a constant $\xi
>0$ such that
\begin{equation}  \label{444}
d(F_{\mathbf{j}},F\backslash F_{\mathbf{j}})\geq \xi |F_{\mathbf{j}}|\text{
for all }\mathbf{j\in \Sigma^{\ast}},
\end{equation}%
and
\begin{equation}  \label{555}
\xi |E_{\mathbf{e}_{\mathbf{j}}}|\leq |F_{\mathbf{j}}|\leq \xi ^{-1}|E_{%
\mathbf{e}_{\mathbf{j}}}| \text{ for all }\mathbf{j\in \Sigma^{\ast}},
\end{equation}%
where we denote by $E_{\mathbf{e}_{\mathbf{j}}}(\subset E)$ the smallest
copy containing $f(F_{\mathbf{j}}).$

\begin{lemma}
\label{l:depth}There is a positive integer $N$ such that for any copy $F_{%
\mathbf{j}}$ of $F$ and smallest copy $E_{\mathbf{e}_{\mathbf{j}}}(\subset
E) $ containing $f(F_{\mathbf{j}}),$ there is a set $\Delta _{\mathbf{j}}$
composed of pathes $\mathbf{e}^{\prime }$ with length $N$ satisfying
\begin{equation*}
f(F_{\mathbf{j}})=\bigcup\nolimits_{\mathbf{e}^{\prime }\in \Delta _{\mathbf{%
j}}}E_{\mathbf{e}_{\mathbf{j}}\mathbf{\ast e}^{\prime }}.
\end{equation*}
\end{lemma}

\begin{proof}
Now let $N=[\frac{\log c^{-1}\xi ^{2}(n-1)^{-1}}{\log \lambda }]+1.$ It
suffices to show that if $z\in E_{\mathbf{e}_{\mathbf{j}}\mathbf{\ast e}%
^{\prime }}$ with $E_{\mathbf{e}_{\mathbf{j}}\mathbf{\ast e}^{\prime }}\cap
f(F_{\mathbf{j}})\neq \emptyset $ then $z\in f(F_{\mathbf{j}}).$ In fact, if
$z\in f(F\backslash F_{\mathbf{j}})$ and $z^{\prime }\in E_{\mathbf{e}_{%
\mathbf{j}}\mathbf{\ast e}^{\prime }}\cap f(F_{\mathbf{j}}),$ by (\ref{444}%
)-(\ref{555}) we have%
\begin{equation*}
|z-z^{\prime }|\geq d(f(F_{\mathbf{j}}),f(F\backslash F_{\mathbf{j}}))\geq
c^{-1}\xi |F_{\mathbf{j}}|\geq c^{-1}\xi ^{2}|E_{\mathbf{e}_{\mathbf{j}}}|.
\end{equation*}%
On the other hand, using (\ref{333}) and the fact that $1=|E|\leq
|E_{v}|\leq n-1$, we have%
\begin{equation*}
|z-z^{\prime }|\leq |E_{\mathbf{e}_{\mathbf{j}}\mathbf{\ast e}^{\prime
}}|\leq \lambda ^{N}(n-1)|E_{\mathbf{e}_{\mathbf{j}}}|<c^{-1}\xi ^{2}|E_{%
\mathbf{e}_{\mathbf{j}}}|,
\end{equation*}%
this is a contradiction.
\end{proof}

For any Borel set $B\subset F,$ we let
\begin{equation*}
h(B)=\frac{\mathcal{H}^{s}(f(B))}{\mathcal{H}^{s}(B)}.
\end{equation*}
Since $f:F\rightarrow E$ is bi-Lipschitz, we have
\begin{equation*}
d=\sup\limits_{\mathbf{j}\in \Sigma ^{\ast }}h(F_{\mathbf{j}})<\infty.
\end{equation*}

\begin{lemma}
\label{l:finite}There is a finite set $\Lambda $ such that
\begin{equation*}
\frac{h(F_{\mathbf{j}\ast j})}{h(F_{\mathbf{j}})}\in \Lambda
\end{equation*}%
for all $\mathbf{j}\in \Sigma ^{\ast }$ and all $j\in \{1,\cdots ,t\}.$
\end{lemma}

\begin{proof}
We note that%
\begin{equation*}
\frac{h(F_{\mathbf{j}\ast j})}{h(F_{\mathbf{j}})}=\frac{\mathcal{H}^{s}(f(F_{%
\mathbf{j}\ast j}))/\mathcal{H}^{s}(F_{\mathbf{j}\ast j})}{\mathcal{H}%
^{s}(f(F_{\mathbf{j}}))/\mathcal{H}^{s}(F_{\mathbf{j}})}=\frac{\mathcal{H}%
^{s}(F_{\mathbf{j}})}{\mathcal{H}^{s}(F_{\mathbf{j}\ast j})}\cdot \frac{%
\mathcal{\lambda }^{s|\mathbf{e}_{\mathbf{j}\ast j}|}}{\mathcal{\lambda }^{s|%
\mathbf{e}_{\mathbf{j}}|}}\cdot \frac{\mathcal{H}^{s}(f(F_{\mathbf{j}\ast
j}))/\mathcal{\lambda }^{s|\mathbf{e}_{\mathbf{j}\ast j}|}}{\mathcal{H}%
^{s}(f(F_{\mathbf{j}}))/\mathcal{\lambda }^{s|\mathbf{e}_{\mathbf{j}}|}}.
\end{equation*}%
Now, $\frac{\mathcal{H}^{s}(F_{\mathbf{j}})}{\mathcal{H}^{s}(F_{\mathbf{j}%
\ast j})}\in \{(r_{j})^{-s}\}_{j=1}^{t}.$ Suppose $M$ is a upper bound for
difference of lengths of $\mathbf{e}_{\mathbf{j}\ast j}$ and $\mathbf{e}_{%
\mathbf{j}},$ we have%
\begin{equation*}
\frac{\mathcal{\lambda }^{s|\mathbf{e}_{\mathbf{j}\ast j}|}}{\mathcal{%
\lambda }^{s|\mathbf{e}_{\mathbf{j}}|}}\in \{\lambda ^{sk}:k\leq M\}
\end{equation*}%
which is a finite set. By Lemma \ref{l:depth}, we also obtain that%
\begin{eqnarray*}
\frac{\mathcal{H}^{s}(f(F_{\mathbf{j}}))}{\mathcal{\lambda }^{s|\mathbf{e}_{%
\mathbf{j}}|}} &=&\frac{\sum_{\mathbf{e}^{\prime }\in \Delta _{\mathbf{j}}}%
\mathcal{H}^{s}(E_{\mathbf{e}_{\mathbf{j}}\mathbf{\ast e}^{\prime }})}{%
\mathcal{\lambda }^{s|\mathbf{e}_{\mathbf{j}}|}} \\
&=&\mathcal{\lambda }^{s(|\mathbf{e}_{\mathbf{j}}|+N)}\frac{\sum_{\mathbf{e}%
^{\prime }\in \Delta _{\mathbf{j}}}\mathcal{H}^{s}(E_{b(\mathbf{e}^{\prime
})})}{\mathcal{\lambda }^{s|\mathbf{e}_{\mathbf{j}}|}} \\
&\in &\lambda ^{sN}\left\{ \sum_{\mathbf{e}^{\prime }\in \Delta }\mathcal{H}%
^{s}(E_{b(\mathbf{e}^{\prime })}):\Delta \subset \{\mathbf{e}^{\prime }:|%
\mathbf{e}^{\prime }|=N\}\right\}
\end{eqnarray*}%
which is also a finite set.
\end{proof}

\begin{lemma}
\label{l:linear}There is a copy $F_{j_{1}\cdots j_{k^{\ast }}}$ of $F$ and a
constant $\bar{d}>0$ such that
\begin{equation}
\frac{\mathcal{H}^{s}(f(B))}{\mathcal{H}^{s}(B)}=\bar{d}  \label{vvv}
\end{equation}%
for Borel set $B\subset F_{j_{1}\cdots j_{k^{\ast }}}.$
\end{lemma}

\begin{proof}
Suppose $\alpha=\max_{x\in (-\infty ,1)\cap \Lambda }x<1\ $or $\alpha=1/2$
if $(-\infty ,1)\cap \Lambda =\emptyset .$ Take $\epsilon >0$ such that
\begin{equation}
\max_{i}(\alpha r_{i}^{s}+(1+\epsilon )(1-r_{i}^{s}))<1.  \label{aaa}
\end{equation}

Let $d=\sup\limits_{\mathbf{j}\in \Sigma ^{\ast }}h(F_{\mathbf{j}})<\infty $
and take a sequence $\mathbf{j=}j_{1}\cdots j_{k^{\ast }}$ such that $\frac{d%
}{h(F_{\mathbf{j}})}<1+\epsilon $. We notice that
\begin{equation*}
\bar{d}\hat{=}h(F_{\mathbf{j}})=\sum_{j}\frac{\mathcal{H}^{s}(F_{\mathbf{j}%
\ast j})}{\mathcal{H}^{s}(F_{\mathbf{j}})}h(F_{\mathbf{j}\ast j})\text{ with
}\sum_{j}\frac{\mathcal{H}^{s}(F_{\mathbf{j}\ast j})}{\mathcal{H}^{s}(F_{%
\mathbf{j}})}=\sum_{j}(r_{j})^{s}=1,
\end{equation*}%
i.e., we have%
\begin{equation}
1=\sum_{j}(r_{j})^{s}\frac{h(F_{\mathbf{j}\ast j})}{h(F_{\mathbf{j}})}\text{
with }\sum_{j}(r_{j})^{s}=1,  \label{ccc}
\end{equation}

We will first show that $h(F_{\mathbf{j}\ast j})\geq h(F_{\mathbf{j}})$ for
all $j.$ Otherwise, without loss of generality, we assume that $\frac{h(F_{%
\mathbf{j}\ast 1})}{h(F_{\mathbf{j}})}<1$. Then
\begin{equation*}
\frac{h(F_{\mathbf{j}\ast 1})}{h(F_{\mathbf{j}})}\leq \alpha\text{ and }%
\frac{h(F_{\mathbf{j}\ast j})}{h(F_{\mathbf{j}})}\leq \frac{d}{h(F_{\mathbf{j%
}})}<1+\epsilon \text{ for }j\geq 2.
\end{equation*}%
It follows from (\ref{aaa}) that%
\begin{equation*}
1=\sum_{j}(r_{j})^{s}\frac{h(F_{\mathbf{j}\ast j})}{h(F_{\mathbf{j}})}\leq
\alpha r_{1}^{s}+(1+\epsilon )(1-r_{1}^{s})<1,
\end{equation*}%
this is a contradiction. Now $h(F_{\mathbf{j}\ast j})\geq h(F_{\mathbf{j}})$
for all $j,$ by (\ref{ccc}) we obtain that
\begin{equation*}
h(F_{\mathbf{j}\ast j})=h(F_{\mathbf{j}})=\bar{d} \text{ for all }j.
\end{equation*}

In the same way, we have
\begin{equation*}
h(F_{\mathbf{j}\ast j_{1}\ast j_{2}})=h(F_{\mathbf{j}})=\bar{d}\text{ for
all }j_{1},j_{2}.
\end{equation*}%
Again and again, we obtain
\begin{equation*}
h(F_{\mathbf{j}^{\prime }})=\bar{d} \text{ for any }\mathbf{j}^{\prime }
\text{ with prefix } \mathbf{j}.
\end{equation*}
Then (\ref{vvv}) follows.
\end{proof}

\begin{proof}[Proof of Proposition \protect\ref{P:1}]
$\ $

Take $\mathbf{j=}j_{1}\cdots j_{k^{\ast }}$ in Lemma \ref{l:linear}. For any
$j,$ we consider the sequence $\mathbf{j}[j]^{k}=\mathbf{j}\ast [j]^{k}$
where the sequence $[j]^{k} $ is composed of $k$ successive digits $j.$ Then
\begin{equation*}
\frac{h(F_{\mathbf{j}[j]^{k^{\prime }}})}{h(F_{\mathbf{j}[j]^{k}})}=1\text{
with }k>k^{\prime }.
\end{equation*}
Hence we obtain that
\begin{eqnarray*}
(r_{j}^{s})^{k-k^{\prime }}=\frac{\mathcal{H}^{s}(F_{\mathbf{j}[j]^{k}})}{%
\mathcal{H}^{s}(F_{\mathbf{j}[j]^{k\prime }})} &=&\frac{h(F_{\mathbf{j}%
[j]^{k^{\prime }}})}{h(F_{\mathbf{j}[j]^{k}})}\cdot \frac{\sum_{\mathbf{e}%
^{\prime }\in \Delta _{\mathbf{j}[j]^{k}}}\mathcal{H}^{s}(E_{b(\mathbf{e}%
^{\prime })})}{\sum_{\mathbf{e}^{\prime }\in \Delta _{\mathbf{j}%
[j]^{k^{\prime }}}}\mathcal{H}^{s}(E_{b(\mathbf{e}^{\prime })})}\cdot
\lambda ^{s(|\mathbf{e}_{\mathbf{j}[j]^{k}}|-|\mathbf{e}_{\mathbf{j}%
[j]^{k^{\prime }}}|)} \\
&=&\frac{\sum_{\mathbf{e}^{\prime }\in \Delta _{\mathbf{j}[j]^{k}}}\mathcal{H%
}^{s}(E_{b(\mathbf{e}^{\prime })})}{\sum_{\mathbf{e}^{\prime }\in \Delta _{%
\mathbf{j}[j]^{k^{\prime }}}}\mathcal{H}^{s}(E_{b(\mathbf{e}^{\prime })})}%
\cdot \lambda ^{s(|\mathbf{e}_{\mathbf{j}[j]^{k}}|-|\mathbf{e}_{\mathbf{j}%
[j]^{k^{\prime }}}|)}.
\end{eqnarray*}%
From the finiteness, we can find $k\neq k^{\prime }$ such that $\Delta _{%
\mathbf{j}[j]^{k}}=\Delta _{\mathbf{j}[j]^{k\prime }}$ then
\begin{equation*}
(r_{j}^{s})^{k-k^{\prime }}=\lambda ^{s(|\mathbf{e}_{\mathbf{j}[j]^{k}}|-|%
\mathbf{e}_{\mathbf{j}[j]^{k^{\prime }}}|)},
\end{equation*}%
that means $(r_{j})^{k-k^{\prime }}=\lambda ^{|\mathbf{e}_{\mathbf{j}%
[j]^{k}}|-|\mathbf{e}_{\mathbf{j}[j]^{k^{\prime }}}|},$ i.e.,
\begin{equation*}
\log r_{j}/\log \lambda \in \mathbb{Q}
\end{equation*}%
for all $j.$ Then Proposition 1 is proved.
\end{proof}

\bigskip

\section{Proof of Theorem}

\subsection{Dimension polynomials}

\

From \cite{jiang} we have
\begin{equation*}
P(x)=x^{2}-nx+m=(x-\beta )(x-\beta ^{\prime })\text{ with }|\beta ^{\prime
}|<1<\beta .
\end{equation*}%
Using notations in Proposition \ref{P:1}, we consider the following two
polynomials
\begin{equation}
\bar{P}(x)=P(x^{k})\text{ and }\bar{Q}(x)=x^{k_{t}}-%
\sum_{i=1}^{t}x^{k_{t}-k_{i}}.  \label{999}
\end{equation}

\begin{proposition}
\label{P:2}Let $s=\dim _{H}E=\dim _{H}F$ and $r$ the ratio in Proposition %
\ref{P:1}. Then%
\begin{equation*}
\bar{P}(r^{-s})=\bar{Q}(r^{-s})=0.
\end{equation*}
\end{proposition}

\begin{proof}
It follows from \cite{jiang} that for $s=\dim _{H}E,$%
\begin{equation*}
(\lambda ^{-s})^{2}-n(\lambda ^{-s})+m=0.
\end{equation*}%
On the other hand, for $s=\dim_{H} F,$ by the \textbf{SSC} we have
\begin{equation*}
\sum\nolimits_{i=1}^{t}(r_{i})^{s}=1.
\end{equation*}%
Then the proposition follows the relations in Proposition \ref{P:1}.
\end{proof}

\subsection{Irreducibility of polynomial}

\begin{proposition}
\label{P:4} For any $Q\in \{x^{p}-\sum_{i=0}^{p-1}b_{i}x^{i}:p\geq
1,b_{i}\in \mathbb{Z}$ and $b_{i}\geq 0\},$ we have
\begin{equation*}
P(x^{q})\nmid Q(x)
\end{equation*}
\end{proposition}

\begin{proof}
Let $Q(x)=(\sum a_{i}x^{i})(x^{2q}-nx^{q}+m)$. Suppose%
\begin{equation*}
\sum a_{i}x^{i}=P_{0}+P_{1}+\cdots +P_{q-1}
\end{equation*}%
where $P_{v}=\sum\limits_{i\equiv v(\text{mod}q)}a_{i}x^{i}\ $for $%
v=0,1,\cdots ,(q-1).$

Then we have
\begin{equation*}
Q(x)=P_{0}P(x^{q})\oplus P_{1}P(x^{q})\oplus \cdots \oplus P_{q-1}P(x^{q}),
\end{equation*}%
where $\oplus $ means the orthogonality of above polynomials in the basis $%
\{1,x,x^{2},\cdots \}.$

Without loss of generality, we assume that
\begin{equation*}
\deg (\sum a_{i}x^{i})\equiv u\text{ (mod}q\text{) with }0\leq u\leq q-1%
\text{.}
\end{equation*}

Let $c_{i}=a_{qi+u},$ then
\begin{equation*}
P_{u}=x^{u}(c_{0}+c_{1}x^{q}+c_{2}x^{2q}+\cdots +c_{l}x^{lq})=x^{u}U(x^{q}).
\end{equation*}%
Since $p\equiv 2q+\deg (\sum a_{i}x^{i})\equiv u$ (mod$q$), we have%
\begin{equation*}
x^{u}U(x^{q})P(x^{q})=x^{p}-\sum_{j\equiv u(\text{mod}q)}b_{j}x^{j},
\end{equation*}%
which implies
\begin{equation*}
U(x)P(x)=x^{p^{\prime }}-\sum_{i=0}^{p^{\prime }}b_{i}^{\prime }x^{i}\text{
with }b_{i}^{\prime }\in \mathbb{Z}\text{ and }b_{i}^{\prime }\geq 0.
\end{equation*}%
Therefore we obtain that
\begin{equation*}
(x^{2}-nx+m)(c_{0}+c_{1}x+c_{2}x^{2}+\cdots
+c_{l}x^{l})=x^{l+2}-\sum_{i=0}^{l+1}b_{i}^{\prime }x^{i},
\end{equation*}%
where
\begin{equation}
c_{l}=1.  \label{1}
\end{equation}%
We recall that
\begin{equation*}
x^{2}-nx+m=(x-\beta )(x-\beta ^{\prime })\text{ with }\beta >1>|\beta
^{\prime }|.
\end{equation*}

Now, we have the following

\begin{claim}
For any $0\leq i\leq l-1,$%
\begin{equation}
c_{i+1}\leq c_{i}\beta ^{-1}\leq 0.  \label{indu}
\end{equation}
\end{claim}

We will verify (\ref{indu}) by induction.

(1) For $i=0,$ we have $c_{0}m=-b_{0}^{\prime }\leq 0$ and thus $c_{0}\leq
0. $

(2) For $i=1,$ we have $-c_{0}n+mc_{1}=-b_{1}^{\prime }\leq 0$ and thus
\begin{equation*}
c_{1}\leq \frac{n}{m}c_{0}\leq \beta ^{-1}c_{0}\leq 0
\end{equation*}%
here $\frac{n}{m}>1>\beta ^{-1}.$

(3) Assume that (\ref{indu}) is true for $i-1,$ i.e., we have $c_{i}\leq
c_{i-1}\beta ^{-1}\leq 0.$ Hence
\begin{equation*}
mc_{i+1}-nc_{i}+\beta c_{i}\leq mc_{i+1}-nc_{i}+c_{i-1}=-b_{i+1}^{\prime
}\leq 0,
\end{equation*}%
which implies%
\begin{equation*}
mc_{i+1}\leq \frac{(n-\beta )}{m}c_{i}=\beta ^{-1}c_{i}\leq 0
\end{equation*}%
due to $\frac{(n-\beta )}{m}=\beta ^{-1}$. Then (\ref{indu}) is verified. In
particular, we have
\begin{equation*}
c_{l}\leq 0
\end{equation*}
which contradicts to (\ref{1}).
\end{proof}

\begin{proposition}
\label{P:3}Suppose $m\notin \{a^{i}\ |\ a\in \mathbb{N}$ and $i\in \mathbb{N}
$ with $i\geq 2\}.$ Then
\begin{equation*}
P(x^{q})\text{ is irreducible in }\mathbb{Z}[x]\text{ for any }q\geq 1.
\end{equation*}
\end{proposition}

\begin{proof}
Note that $P(x)=P(x^{1})$ is irreducible (e.g. see \cite{jiang}). Without
loss of generality, we assume that $q\geq 2.$

Let $\omega =e^{2\pi \sqrt{-1}/q}.$ Then
\begin{equation*}
P(x^{q})=\left( \prod\nolimits_{i=0}^{q-1}(x-\omega ^{i}\beta ^{1/q})\right)
\cdot \left( \prod\nolimits_{i=0}^{q-1}(x-\omega ^{i}(\beta ^{\prime
})^{1/q})\right) .
\end{equation*}%
Suppose on the contrary that $P(x^{q})=Q_{1}(x)Q_{2}(x)$ and $%
Q_{1}(x),Q_{2}(x)\in \mathbb{Z}[x]$ with $\deg Q_{1},\deg Q_{2}\geq 1.$ We
note that
\begin{equation*}
m=|P(0)|=|Q_{1}(0)|\cdot |Q_{2}(0)|,
\end{equation*}%
where
\begin{equation*}
|Q_{1}(0)|=|\beta ^{u_{1}}(\beta ^{\prime })^{v_{1}}|^{1/q}\in \mathbb{N}%
\text{ and }|Q_{2}(0)|=|\beta ^{u_{2}}(\beta ^{\prime })^{v_{2}}|^{1/q}\in
\mathbb{N}
\end{equation*}%
with $u_{1},v_{1},u_{2},v_{2}\geq 1.$

We will show that $u_{1}=v_{1}.$ Otherwise by symmetry we may assume that $%
u_{1}>v_{1},$ then
\begin{equation*}
(\beta ^{u_{1}-v_{1}})=\frac{|Q_{1}(0)|^{q}}{(\beta \beta ^{\prime })^{v_{1}}%
}=\frac{|Q_{1}(0)|^{q}}{(m)^{v_{1}}},
\end{equation*}%
which implies
\begin{equation*}
R(\beta )=0\text{ with }R(x)=m^{v_{1}}x^{u_{1}-v_{1}}-|Q_{1}(0)|^{q}\in
\mathbb{Z}[x].
\end{equation*}%
By \cite{jiang}, we obtain that $P(x)=x^{2}-nx+m$ is an irreducible
polynomial satisfying $P(\beta )=0.$ Therefore, we have
\begin{equation*}
P|R\text{ but }R\text{ only has roots with module }\beta .
\end{equation*}%
Now $R(\beta ^{\prime })=P(\beta ^{\prime })=0$ with $|\beta ^{\prime
}|<|\beta |.$ This is a contradiction.

In the same way, we have $u_{2}=v_{2}.$ Now we obtain that
\begin{equation*}
u_{1}=v_{1}\text{ and }u_{2}=v_{2}.
\end{equation*}%
Let $u_{1}/q=j/i$ with $(i,j)=1$ and $j<i$ ($i\geq 2$), then $%
u_{2}/q=(i-j)/i $ since $u_{1}+u_{2}=q.$ Hence
\begin{equation*}
|Q_{1}(0)|=m^{\frac{j}{i}}\in \mathbb{N}\text{ and }|Q_{2}(0)|=m^{\frac{i-j}{%
i}}\in \mathbb{N}
\end{equation*}%
and thus $m^{\frac{1}{i}}=a\in \mathbb{N}$ and $m=a^{i}$ with $i\geq 2.$
This is a contradiction.
\end{proof}

\subsection{Proof of Theorem}

$\ $\

It follows from Propositions \ref{P:1}-\ref{P:2} that there are $r\in (0,1)$
and $k,k_{1}\leq k_{2}\leq \cdots \leq k_{t}\in \mathbb{N}$ such that
\begin{equation*}
\bar{P}(r^{-s})=\bar{Q}(r^{-s})=0,
\end{equation*}%
where $\bar{P}$ and $\bar{Q}$ are defined in (\ref{999}).

Suppose on the contrary that $\bar{P}(x)=P(x^{k})=x^{2k}-nx^{k}+m$ is
irreducible in $\mathbb{Z}[x]$, then we have
\begin{equation*}
P(x^{k})|(x^{k_{t}}-\sum_{i=1}^{t}x^{k_{t}-k_{i}}),
\end{equation*}%
which contradicts to Proposition \ref{P:4}. Therefore $P(x^{k})\text{ is
reducible in }\mathbb{Z}[x], $ and thus $m\in \{a^{i}\ |\ a\in \mathbb{N}$
and $i\in \mathbb{N} $ with $i\geq 2\}$ by Proposition \ref{P:3}.

\end{document}